\documentclass[10pt]{article}
\usepackage{amsmath}
\usepackage[all]{xy}
\usepackage{amssymb} 
\usepackage{amsthm}
\usepackage{graphicx}
\begin{document}  
\title{Space of group orderings, quasi morphisms and bounded cohomology}
\author{Tetsuya Ito}
\date{}

 
\newtheorem{thm}{Theorem}
\newtheorem{cor}{Corollary}
\newtheorem{lem}{Lemma}
\newtheorem{prop}{Proposition}
 {\theoremstyle{definition}
  \newtheorem{defn}{Definition}}
 {\theoremstyle{definition}
  \newtheorem{exam}{Example}}
 {\theoremstyle{definition}
  \newtheorem{rem}{Remark}}
 \newcommand{\Z}{ \mathbb{Z}}
 \newcommand{\Q}{ \mathbb{Q}}
 \newcommand{\R}{ \mathbb{R}}
 \newcommand{\mU}{ \mathcal{U}}
  \newcommand{\InvLO}{ \textrm{InvLO}\,}
  \newcommand{\LO}{ \textrm{LO}\,}
    \newcommand{\BO}{ \textrm{BO}\,}
  \newcommand{\Cof}{ \textrm{Cof}\,}
  \newcommand{\Hom}{\textrm{Hom}}
 
\maketitle
\begin{abstract}
For a group $G$, we construct a quasi morphism from its left orderings and the map from the space of left orderings to the second bounded cohomology. We show that these maps reflect various properties of the group orderings.   
\end{abstract}

\section{Introduction}

A total ordering $<$ of a group $G$ is called a {\it left ordering} if the relation $<$ is preserved by the left action of $G$, that is, $a<b$ implies $ca<cb$ for all $c \in G$. The right orderings are defined by the similar way. An ordering $<$ is called a {\it bi-ordering} if the ordering is both left and right ordering.
We denote by $\LO(G)$, $\BO(G)$ the set of all left orderings and bi-orderings of $G$, respectively.
Each left ordering $<$ is determined by its positive cone $P_{<} =\{ g \in G \: | \: g > 1\}$,  which have the following two properties.
\begin{description}
\item[LO1] $P$ is stable under the multiplication: $P \cdot P \subset P$.
\item[LO2] $P,P^{-1}$ and $\{1\}$ define a partition of $G$: $G = P \coprod P^{-1} \coprod \{1\}$.
\end{description}
Conversely, for a subset $P$ of $G$ having the properties {\bf [LO1]} and {\bf [LO2]}, the ordering $<_{P}$ defined by $x <_{P} y$ if $x^{-1}y \in P$ is a left-ordering whose positive cone is $P$. Thus, $\LO(G)$ is identified with the subset of the power sets $\{0,1\}^{G}$.  

We define the topology of $\LO(G)$ by giving a basis of open sets of the form
\[  \mU_{g} = \{P \in \LO(G)\: | \: 1<_{P} g \} \;\; (g \in G) \]
It is known that $\LO(G)$ is compact, totally discontinuous \cite{s}.
  
  The aim of this paper is to study the relationship between the space $\LO(G)$ and quasi morphisms or bounded cohomology groups. For $x \in G$, a subgroup $H \subset G$, and $P \in \LO(G)$ having some additional properties, we construct a $\Z$-valued quasi morphism $\rho_{x,P}^{H}:H \rightarrow \Z$ and a map $\psi_{x}^{H}$ from the subspace of $\LO(G)$ to the $2$nd bounded cohomology group of $H$.
  We determine the image of $\psi_{x}^{G}$ for a group whose $\R$-coefficient 2nd bounded cohomology vanishes (Theorem \ref{thm:surj}), and show various properties of the maps $\rho_{x,P}^{H}$ and $\psi_{x}^{H}$.
  
\begin{itemize} 
  \item $\psi_{x}^{G}$ is continuous for a group whose $\R$-coefficient 2nd bounded cohomology vanishes (Theorem \ref{thm:continuous}).
  \item $\psi_{x}^{H}$ classifies the dynamics of $H$ if $x$ belongs to $Z_{G}(H)$, the centralizer of $H$ in $G$ (Theorem \ref{thm:clasify}).
  \item $\rho_{x,P}^{H}$ provides a condition on abelian subgroups to be convex. (Theorem \ref{thm:convex})

\end{itemize} 
  
 The plan of the paper is as follows. In section 2 we describe the construction of quasi morphisms and its fundamental properties. In section 3 we provide a cohomological description of the map $\psi_{x}^{H}$ and study the dynamics of the group $G$. 
 Section 4 we use the map $\psi_{x}^{H}$ to study the convex subgroups. \\

\textbf{Acknowledgments.}
 The author gratefully acknowledges the helpful suggestions of professor Toshitake Kohno during the preparation of the paper. He also wishes to express his thanks to Dale Rolfsen and Adam Clay for stimulating discussions. This research was supported by JSPS Research Fellowships for Young Scientists.

\section{Quasi morphisms derived from left-orderings}

First of all we review the definitions of quasi morphism and bounded cohomology which is used in this paper.

Let $K = \Z$ or $\R$. For a group $G$, the bounded cohomology group $\widehat{H_{b}^{*}}(G;K)$ is a cohomology of the bounded cochain complex. In this paper we only treat the $2$nd bounded cohomology classes which are trivial in the usual group cohomology. Let us denote by $H_{b}^{2}(G;K)$ the kernel of the comparison map $\iota: \widehat{H_{b}^{2}}(G;K) \rightarrow H^{2}(G;K)$. We simply refer $H_{b}^{2}(G;K)$ the {\it 2nd bounded cohomology} of $G$. 
 
 A map $\phi: G \rightarrow K$ is called a {\it quasi morphism} of defect $D$ if 
\[ | \phi(g g') - \phi(g) - \phi(g')| \leq D \]
holds for all $g,g' \in G$. A quasi morphism is called an {\it almost homomorphism} if there exists a homomorphism $\tau: G \rightarrow K$ and a constant $C>0$ such that $|\phi(g) - \tau(g)| \leq C$ holds for all $g \in G$.
 Let $\textrm{QMor}(G;K)$ and $\textrm{AHom}(G;K)$ be the set of quasi morphisms and almost homomorphisms, respectively. 
The following lemma is well-known, but to provide an explicit correspondence we include here the proof.

\begin{lem}
\label{lem:boundhom}
\hspace{5cm}
 \begin{enumerate}
 \item $\textrm{QMor}(G,K) \slash \textrm{AHom}(G,K) \cong  H_{b}^{2}(G;K)$.
 \item Let $i: H_{b}^{2}(G;\Z) \rightarrow H_{b}^{2}(G;\R)$ be the natural map. Then $\textrm{Ker}\, i = H^{1}(G;\R) \slash H^{1}(G;\Z)$. In particular, if $H_{b}^{2}(G;\R) = 0$, then $H_{b}^{2}(G;\Z) \cong H^{1}(G;\R) \slash H^{1}(G;\Z)$.
 \end{enumerate} 
 \end{lem}
\begin{proof}
 By definition, for $\phi \in \textrm{QMor}(G,K)$, $d\phi$ is a bounded $2$-cocycle which represents an element in $H^{2}_{b}(G;K)$. Conversely, for a bounded $2$-cocycle $c: G\times G \rightarrow K$ which represents an element of $H^{2}_{b}(G;K)$,  there exists a 1-cochain $\phi: G \rightarrow K$ such that $c = d\phi$. since $c$ is bounded, $|c(x,y)| = |d\phi(x,y)| = |\phi(x) -\phi(xy) + \phi(y) | \leq C$ for some constant $C$. Thus, $\phi$ is a quasi morphism. 
 
 Now two bounded cocycles $c= d\phi$ and $c'=d\phi'$ represent the same bounded cohomology class if and only if $c-c' = d(\phi - \phi') = d \psi$ holds for some bounded $1$-cochain $\psi$, so $\tau = \phi-\phi'-\psi: G \rightarrow K$ is a homomorphism.
 Therefore, $c= d\phi$ and $c'=d\phi'$ represents the same bounded cohomology class if and only if $\phi - \phi'$ is an almost homomorphism. This completes the proof of (1).

Let $c = d \phi$ be a bounded 2-cocycle which represents an element of $\textrm{Ker}\, i$. Since $[c]$ is trivial in $\R$-coefficient, there exists a $\R$-valued bounded 1-cochain $\phi_{\R}: G \rightarrow \R$ such that $c = d\phi_{\R}$.
Now $d(\phi_{\R} - \phi) = 0$, so $\tau_{c}= \phi_{\R}-\phi : G \rightarrow \R$ is homomorphism. Let $c = d\phi = d\phi_{\R}$ and $c' = d \phi' = d \phi'_{\R}$ be cohomologous bounded 2-cocycles.
Then $c-c' = d \psi$ holds for some bounded $\Z$-valued $1$-cochain $\psi$. So both $\phi -\phi' -\psi$ and $\phi_{\R}-\phi'_{\R} -\psi$ are $\Z$-valued homomorphisms.
Thus $\tau_{c} -\tau_{c'} = (\phi_{\R} - \phi'_{\R} -\psi) - (\phi -\phi' -\psi) $ is $\Z$-valued   homomorphisms.

Conversely, for a homomorphism $\tau: G \rightarrow \R$, let $c_{\tau} = d[\tau]$, where $[\tau]: G \rightarrow \Z$ is a $\Z$-valued 1-cochain taking an integer part of $\tau$. Then, $c_{\tau}$ represents an element in $H_{b}^{2}(G;\Z)$. It is easy to see that $c_{\tau}$ and $c_{\tau'}$ are cohomologus if $\tau-\tau'$ is a $\Z$-valued homomorphism.    
\end{proof} 
 
In this paper we often impose the condition that $H_{b}^{2}(G;\R) = 0$.
This condition is satisfied if $G$ is {\it amenable}, for example, $G$ is abelian group, solvable group and so on.

Now we define quasi-morphisms from left orderings, and construct maps mentioned in introduction.  
We say an element $x \in G$ is {\it co-final} to a subset $A \subset G$ with respect to the left ordering $P$ of $G$ if for each element $a \in A$, there exists an integer $N$ such that $x^{-N} <_{P} a <_{P} x^{N}$ holds.
We say $x$ is {\it universally co-final} to $A$ if $x$ is co-final to $A$ with respect to all left orderings of $G$.

 Let $H$ be a subgroup of $G$. For an element $x \in G$, let $\Cof_{x}^{H}(G)$ be the set of left orderings of $G$ such that $x$ is co-final to $H$.
We define $\InvLO_{x}^{H}(G)$ as the set of left orderings of $G$ such that the right action of $x$ on the subgroup generated by $H$ and $x$ preserves the ordering. That is, $P \in \InvLO_{x}^{H}(G)$ if and only if $b <_{P} b'$ implies $bx <_{P} b'x$ for all $b,b'$ in the subgroup generated by $x$ and $H$.

It is obvious that $\BO(G) \subset \InvLO_{x}^{G}(G) $, and $\InvLO_{x}(G) = \LO(G)$ if $x$ is central.
In the case $H=G$, we simply denote $\Cof_{x}^{G}(G)$, $\InvLO_{x}^{G}(G)$ by $\Cof_{x}(G)$, $\InvLO_{x}(G)$ respectively.

\begin{lem}
\label{lem:subset}
\hspace{6cm}
\begin{enumerate}
\item $\InvLO_{x}^{H}(G)$ is closed subset of $\LO(G)$.
\item If $H$ is finitely generated, then $\InvLO_{x}^{H}(G) \cap \Cof_{x}^{H}(G)$ is an open subset of $\InvLO_{x}^{H}(G)$. 
\end{enumerate}
\end{lem}
\begin{proof}

Let $H'$ be the subgroup of $G$ generated by $H$ and $x$.
Then the set $\InvLO_{x}^{H}(G)$ is written by
\[ \InvLO_{x}^{H}(G) = \bigcap_{h' \in H'} (\mU_{h'} \cap \mU_{x^{-1}h'x}). \]
Observe that $\mU_{h}$ is the complement of $\mU_{h^{-1}}$, hence it is closed.
Thus $\InvLO_{x}^{H}(G)$ is closed subset of $\LO(G)$.

Now assume that $H$ is finitely generated, and let $h_{1},\ldots,h_{n}$ be a generator of $H$.
To show (2), we first observe that for $P \in \InvLO_{x}^{H}(G)$, $x$ is co-final to $H$ if and only if $x$ is co-final to the generating subset $\{h_{1},\ldots, h_{n}\}$.
The set $\Cof_{x}^{H}(G) \cap \InvLO_{x}^{H}(G)$ is written as
\begin{eqnarray*}
 \Cof_{x}^{H}(G) \cap  \InvLO_{x}^{H}(G)& = & \left[ \mU'_{x} \cap \left\{ \bigcap_{i=1}^{n}\bigcup_{N\geq 1} (\mU'_{h_{i}^{-1}x^{N}}\cap \mU'_{x^{N}h_{i}})  \right\} \right] \\
 &  & \hspace{0.5cm} \bigcup \left[ \mU'_{x^{-1}} \cap \left\{ \bigcap_{i=1}^{n}\bigcup_{N\geq 1} (\mU'_{h_{i}^{-1}x^{-N}}\cap \mU'_{x^{-N}h_{i}})  \right\} \right]
\end{eqnarray*}
where $\mU'_{g} = \mU_{g} \cap \InvLO_{x}^{H}(G)$.
 Hence $\Cof_{x}^{H}(G) \cap \InvLO_{x}^{H}(G)$ is open in $\InvLO_{x}^{H}(G)$.
\end{proof}
For an ordering $P \in \InvLO_{x}^{H}(G) \cap \Cof_{x}^{H}(G)$,
we define the map $\rho_{x,P}^{H}:H \rightarrow \Z$ by
\[ \rho_{x,P}^{H}(h) = 
\left\{ 
\begin{array}{l}
N \; \textrm{ such that } x^{N} \leq_{P} h <_{P} x^{N+1} \textrm{  if } x >_{P} 1 \\ 
N \; \textrm{ such that } x^{-N-1} <_{P} h \leq_{P} x^{-N} \textrm{  if } x <_{P} 1. \\
\end{array} 
\right.
\]

\begin{lem}
\label{lem:quasi}
$\rho_{x,P}^{H}$ is a quasi morphism of defect $1$.
\end{lem}
\begin{proof}
We prove the case $x >_{P} 1$. The other case is similar.
For  $h,h' \in H$, let $x^{N} \leq_{P} h <_{P} x^{N+1}$ and $x^{M} \leq_{P} h <_{P} x^{M+1}$.
Then, $ hx^{M} \leq_{P} hh' <_{P} hx^{M+1} $. Since the ordering $<_{P}$ is invariant under the right multiplication of $x$, we conclude that $x^{M+N} \leq_{P}  hh' <_{P} x^{M+N+2}$.
\end{proof}

Now let us define the stable map $\overline{\rho^{H}_{x,P}} : H \rightarrow \R$ by 
\[ \overline{\rho^{H}_{x,P}} (h)= \lim_{N \rightarrow \infty} \frac{\rho^{H}_{x,P}(h^{N})}{N}. \] 

From Lemma \ref{lem:quasi}, $\overline{\rho^{H}_{x,P}}$ is a well-defined $\R$-valued quasi morphism of defect one. It is routine to check the following properties.

\begin{lem}
\label{lem:stablemap}
The stable map $\overline{\rho^{H}_{x,P}} : H \rightarrow \R$ has the following properties.
\begin{enumerate}
\item $\overline{\rho^{H}_{x,P}}(a^{-1}ha) = \overline{\rho^{H}_{x,P}}(h)$ holds for all $a,h \in H$.
\item $\overline{\rho^{H}_{x,P}}(h^{M}) = M \overline{\rho^{H}_{x,P}}(h)$ holds for all $h\in H$ and $M \in \Z$.
\item If $\overline{\rho^{H}_{x,P}}(h_{1}h_{2}\cdots h_{k}) = 0$, then $| \overline{\rho^{H}_{x,P}}(h_{1}) + \overline{\rho^{H}_{x,P}}(h_{2}) + \cdots + \overline{\rho^{H}_{x,P}}(h_{k}) | \leq k-1$ holds. 
\end{enumerate}
\end{lem}

The above construction of quasi morphisms is motivated from the following example.

\begin{exam}
Let $B_{n}$ be the braid group of $n$ strands, $\sigma_{1},\ldots,\sigma_{n-1}$ be the standard generators of $B_{n}$, and $\Delta= (\sigma_{1}\sigma_{2}\cdots\sigma_{n-1})\cdots(\sigma_{1}\sigma_{2})(\sigma_{1})$. $\Delta^{2}$ is a generator of the center of $B_{n}$.
It is known $B_{n}$ is left-orderable and $\Delta^{2}$ is universally co-final \cite{ddrw}.
Let $<_{D}$ be the Dehornoy ordering of $B_{n}$, which is the standard left ordering of $B_{n}$. See \cite{ddrw} for precise definition.
The quasi morphism $\rho^{B_{n}}_{\Delta^{2},<_{D}}$ is called the {\it Dehornoy floor quasi morphism}, and the stable map $\overline{\rho^{B_{n}}_{\Delta^{2},<_{D}}}$ is called {\it the twisting number}, which are defined in \cite{ma}. 
These quasi morphisms are quite useful to study the relationships between topology and orderings \cite{i1},\cite{i2},\cite{ma}.
\end{exam}

Now we are ready to define a map from the space of left-orderings to the bounded cohomology groups.

\begin{defn}
Let us define the map $\psi^{H}_{x}: \InvLO_{x}^{H}(G) \cap \Cof_{x}^{H}(G) \rightarrow H_{b}^{2}(H;\Z)$ by
$\psi^{H}_{x}(P) = [\rho_{x,P}^{H}] $.
\end{defn}

The map $\psi_{x}^{H}$ are natural with respect to the inclusions in the following sense.

Let $K$ and $H$ be subgroups of $G$, and assume that $K$ is also a subgroup of $H$.
Take $x \in H$.
Let $i: H \rightarrow G$ and $j : K \rightarrow H$ be inclusions. Then, by taking the restriction, $i$ induces a continuous map $i^{*}: \LO(G) \rightarrow \LO(H)$, which is explicitly written as $i^{*}(P) = P \cap H$. Clearly this map defines the continuous map $i^{*}: \InvLO_{x}^{H}(G) \cap \Cof_{x}^{H}(G) \rightarrow \InvLO_{x}^{K}(H) \cap \Cof_{x}^{K}(H) $. 
Then by definition, it is easy to see that the following diagram commutes.
\[
\xymatrix{ \InvLO_{x}^{H}(G) \cap \Cof_{x}^{H}(G) \ar[r]^{i^{*}} \ar[d]_{\psi_{x}^{H}} & \InvLO_{x}^{K}(H) \cap \Cof_{x}^{K}(H) \ar[d]^{\psi_{x}^{K}} \\
H_{b}^{2}(H;\Z) \ar[r]^{j^{*}}  & H_{b}^{2}(K;\Z) 
} 
\]

Now let us study the properties of $\psi_{x}^{H}$.
 First we show the map $\psi_{x}$ is non-trivial by determining the image of $\psi_{x}$ for  a group $G$ whose $\R$-coefficient $2$nd bounded cohomology vanishes.

\begin{thm}
\label{thm:surj}
Let $G$ be a finitely generated left-orderable group whose $\R$-coefficient $2$nd bounded cohomology vanishes, 
and $x \in G$ be an element of $G$ whose representing homology class $[x] \in H_{1}(G;\Z)$ has infinite order. 
Then the image of the map $\psi_{x}^{G}: \InvLO_{x}(G) \cap \Cof_{x}(G) \rightarrow H_{b}^{2}(G; \Z)$ is given by
\[ \textrm{Im}\, \psi_{x}^{G} = \{ [\tau] \in H_{b}^{2}(G;\Z) = H^{1}(G;\R) \slash H^{1}(G;\Z) \: | \: \tau([x]) = 1\}\]
\end{thm}
\begin{proof}

First we show the theorem for the free abelian group, $G=\Z^{m}$ case.

Let $\{ a_{0}, a_{1},\ldots,a_{m-1} \}$ be a free generator of $\Z^{m}$. The assumption that $[x]$ is infinite order implies that $x$ is a non-trivial element. With no loss of generality,  we can choose $a_{0} = x^{k}$ where $k$ is a positive integer. 

Let $\tau: \Z^{m} \rightarrow \R$ be the homomorphism such that $\tau(x)=1$, and put 
$r_{i} = \tau(a_{i})$ for $i=0,\ldots,m-1$. 
By re-numbering of $\{a_{i}\}$, we may assume that $r_{0},r_{1},\ldots, r_{p} \in \R - \{0\}$, $r_{p+1}= \cdots = r_{m-1} = 0$.
Let $A'$, $A_{0}$ be the subgroup of $A$ generated by $\{a_{0},\ldots,a_{p}\}$, $\{a_{p+1},\ldots,a_{m-1}\}$ respectively.

We define the ordering $P$ of $A'$ as follows.  
Let $L_{\tau} = \{ (t, kr_{1}t , \ldots , kr_{p}t) \in \R^{p+1} \: | \: t \in \R\}$ be the line in $\R^{p+1}$, and $p_{L}: \R^{m} \rightarrow L$ be the orthogonal projection. 
On the points of the line $L_{\tau}$, we define the ordering $<_{L}$ by $(t,\ldots) <_{L} (t',\ldots)$ if $t<_{L} t'$.
We regard $A' = \Z^{p+1}$ as the integer lattice of $\R^{m+1}$.
First we take a left ordering $<'$ of $\textrm{Ker}\,(p_{L}) \cap A' $.
For $a,b \in \Z^{p+1}$, we define $a <_{P} b$ if $p_{L}(a) <_{L} p_{L}(b) $, or $p_{L}(a)=p_{L}(b)$ and $1<' a^{-1}b$. Then $P$ is a left ordering of $\Z^{p+1}$ and $\overline{\rho_{x, P}^{A'}}(a_{i}) = r_{i}$ holds for all $i=0,\ldots,p$.

Next we extend the ordering $P$.
Let $p: A \rightarrow A'$ be the projection, and $Q$ be a left ordering of $A_{0}$. 
We define the ordering $R$ of $A$ by
$a < a'$ if $p(a) <_{P} p(a')$ or, $p(a)=p(a')$ and $a <_{Q} a'$.
From the construction, the restriction of $R$ to $A'$ coincide with the ordering $P$,  and $\overline{\rho_{x, P}^{A}}(A_{0}) = 0$. Hence $\psi_{x}^{A}(R) = [\tau]$, so we complete the proof of theorem for the free abelian groups.

Now we show the general case.
Let $A = H_{1}(G;\Z) - \textrm{Tors} H_{1}(G; \Z)$ be the torsion-free part of the $1$st homology group, and $p :G \rightarrow A$ be the projection. Since $H^{1}(G;\Z) = \textrm{Hom} (A ,\Z)$, the map $p$ induces the isomorphism of the $2$nd bounded cohomologies.
Let $[\widetilde{\tau}]$ be a $2$nd bounded cohomology class of $G$, represented by homomorphism $\widetilde{\tau}: G \rightarrow \R$ such that $\widetilde{\tau}(x)=1$. Then one can find a homomorphism $\tau: A \rightarrow \R$ such that $p^{*}([\tau]) = [\widetilde{\tau}]$.
We construct left-ordering $P$ of $G$ as follows.

First observe that $\tau(x)=1$ and $[x] = p(x)$ is non-trivial element in $H_{1}(G;\Z)$. Thus from the abelian case we have just proved, we can find a left ordering $<_{A}$ of $A$ which represents the bounded cohomology class $[\tau]$.
Now let us consider the exact sequence $1 \rightarrow \textrm{Ker}\, p  \rightarrow G \stackrel{p}{\rightarrow} A \rightarrow 1$. Since $G$ is left orderable, so is $\textrm{Ker}\, p$. Let $<'$ be a left ordering of $\textrm{Ker}\, p$. We define a left ordering $P$ of $G$ by defining $g<_{P} g'$ if $p(g) <_{A} p(g')$, or $p(g)=p(g')$ and $1 <' g^{-1}g'$.
Then $\psi_{x}^{G}(P) = p^{*}([\tau]) = [\widetilde{\tau}]$ holds.
\end{proof}

Next we provide an explicit description of the map $\psi_{x}^{H}$ for a subgroup $H$ whose $\R$ coefficient 2nd bounded cohomology vanishes.

\begin{lem}
\label{lem:explicitform}
Assume that $H_{b}^{2}(H;\R) = 0$.
Let $y_{1},\ldots,y_{b}$ be elements of $G$ which form the basis of $H_{1}(H;\Z) - \textrm{Tors}\, H_{1}(H;\Z) $. Then, for $P \in \InvLO_{x}^{H}(G) \cap \Cof_{x}^{H}(G)$,
\[ \psi_{x}^{H}(P) =  \left[ (\overline{\rho_{x,P}^{H}}(y_{1}), \ldots,\overline{\rho_{x,P}^{H}}(y_{b}) \,) \right] \]
holds.
\end{lem}
\begin{proof}

Recall that for a torsion-free abelian group $\Z^{m} = \langle a_{1},\ldots, a_{m}\rangle$, the isomorphism
\[ \Theta : \textrm{QMor}( \Z^{m} ; \Z) \slash \textrm{AHom}(\Z^{m} ; \Z)  = H_{b}^{2} (\Z^{m} ; \Z) \rightarrow H^{1}(\Z^{m} ;\R) \slash H^{1}(\Z^{m} ;\Z) = \R^{m} \slash \Z^{m} \]
 is given by $\Theta( [\phi] ) = \left[ (\overline{\phi}(a_{1}), \ldots,  \overline{\phi}(a_{m})) \right] $. Here $\overline{\phi} : \Z^{m} \rightarrow \R$ is a stable map of a quasi morphism $\phi$, defined by $\overline{\phi}(a) = \lim_{N \rightarrow \infty} \phi (Na) \slash N$.

Now as in the proof of Theorem \ref{thm:surj}, let $A$ be the torsion-free part of the $1$st homology group of $H$ and $ p: H \rightarrow A$ be the projection.
Since $p^{*}$ induces an isomorphism of bounded cohomology, we can find a quasi morphism $\phi: \Z^{m} \rightarrow \Z$ such that $p^{*}\phi = \phi \circ p$ represents the bounded cohomology class $[\psi_{x}^{H}(P)]$. This implies that the difference $\rho_{x,P}^{H} - \phi\circ p$ is an almost homomorphism.
Thus, there exists a homomorphism $\tau: H \rightarrow \Z$ and constant $C>0$ such that 
\[ \left| \rho_{x,P}^{H}(h) - \phi \circ p(h) - \tau (h) \right| \leq  C \]
holds for all $h \in H$.
Thus 
\[ \left| \frac{ \rho_{x,P}^{H}(h^{N})}{N} - \frac{\phi \circ p(h^{N})}{N} - \tau (h) \right| \leq  \frac{C}{N} \]
holds for all $N>0$ and $h \in H$.
Therefore, the stable maps $\overline{\rho_{x,P}^{H}}$ and $\overline{\phi \circ p}$ coincide modulo $\Z$.

Hence we conclude that 
\[ \psi_{x}^{H}(P) =  \left[ (\overline{\rho_{x,P}^{H}}(y_{1}), \ldots,\overline{\rho_{x,P}^{H}}(y_{b}) ) \right] \]
holds.
\end{proof}

Based on this description, we extend the maps $\psi_{x}^{H}$ for the whole of $\InvLO_{x}^{H}(G)$ so that it contains more information.

\begin{defn}
Let $H$ be a finitely generated subgroup of $G$ whose $\R$-coefficient 2nd bounded cohomology vanishes and $x \in G$. Let $y_{1},\ldots,y_{b}$ be elements of $G$ which form the basis of $H_{1}(H;\Z) - \textrm{Tors}\, H_{1}(H;\Z) $ where $b= b_{1}(H)$.
We define a map $\widetilde{\psi_{x}^{H}} : \InvLO_{x}^{H}(G) \rightarrow S^{b} = \R^{b} \cup \{\infty \}$ 
by
\[
 \widetilde{\psi_{x}^{H}}(P) = \left\{
\begin{array}{ll}
(\overline{\rho_{x,P}^{H}}(y_{1}), \ldots,\overline{\rho_{x,P}^{H}}(y_{b}) )  & (P \in \InvLO_{x}^{H}(G) \cap \Cof_{x}^{H}(G)) \\
\infty & (P \not \in \InvLO_{x}^{H} )
\end{array}
\right.
\]
\end{defn}

The restriction of $\widetilde{\psi_{x}^{H}}$ to $\InvLO_{x}^{H}(G) \cap \Cof_{x}^{H}(G))$ is nothing but the lift of the map  $\psi_{x}^{H} :  \InvLO_{x}^{H}(G) \cap \Cof_{x}^{H}(G)  \rightarrow H_{b}^{2}(H;\Z) = (S^{1})^{b}$.
We mainly use the following special case, which corresponds to the case $H =\Z = \langle y \rangle$.

\begin{defn}
For $y \in H$, let us define the map
\[ \psi^{H}_{x,y}: \InvLO_{x}^{H}(G) \cap \Cof_{x}^{H}(G) \rightarrow \R\slash\Z\]
by $\psi^{H}_{x,y} = \iota^{*}_{y}\circ\psi^{H}_{x}$, where $\iota^{*}_{y} : H^{2}_{b}(H;\Z) \rightarrow H^{2}_{b}(\Z;\Z) \cong \R \slash \Z$ is a map induced by the inclusion $\iota_{y}: \langle y \rangle \hookrightarrow H$. This map is explicitly written by
$\psi^{H}_{x,y}(P) = [\overline{\rho^{H}_{x,P}}(y) ]$.
Finally, we define the map $\widetilde{\psi^{H}_{x,y}}: \InvLO_{x}^{H}(G) \rightarrow S^{1} = \R \cup \{\infty\}$ by 
\[ \widetilde{\psi^{H}_{x,y}}(P) =\left\{
\begin{array}{l}
 \overline{\rho_{x,P}^{H}}(y) \;\;\;\;\; (P \in \Cof_{x}^{H}(G) )\\
 \infty \;\;\;\;\;\;\;\;\;\;\;\;\;( P  \not \in \Cof_{x}^{H}(G))
 \end{array} \right.
\]
This map is an extension of the lift of the map $\psi^{H}_{x,y}$.
\end{defn}

The useful fact is that the maps defined above are continuous.

\begin{prop}
\label{prop:continuous}
Both 
$\psi^{H}_{x,y}$ and $\widetilde{\psi^{H}_{x,y}}$ are continuous.
\end{prop}
\begin{proof}
Since $\widetilde{\psi^{H}_{x,y}}$ is an extension of the lift of $\psi^{H}_{x,y}$, it is sufficient to show that $\widetilde{\psi^{H}_{x,y}}$ is continuous. 
Let $A = \InvLO_{x}^{H}(G) \cap \Cof_{x}^{H}(G)$. First we show that $\widetilde{\psi^{H}_{x,y}}$ is continuous on $A$.
Since $\rho_{x,P}^{H}$ is a quasi morphism of defect $1$, for every $\varepsilon > 0$, there exists an integer $N_{\varepsilon}$ such that 
\[ \left| \, \frac{\rho_{x,P}^{H}(y^{N})}{N} - \widetilde{\psi^{H}_{x,y}}(P) \right| < \varepsilon \, \]
holds for all $N > N_{\varepsilon}$ and $P \in \InvLO_{x}^{H}(G) \cap \Cof_{x}^{H}(G)$.

We show that $ \widetilde{\psi^{H}_{x,y}} {}^{-1}((a-\varepsilon, a+ \varepsilon))$ is open for all $a \in \R$ and $\varepsilon >0$.
Let $P \in \widetilde{\psi^{H}_{x,y}} {}^{-1}((a-\varepsilon, a+ \varepsilon))$ and $a^{*} = \widetilde{\psi^{H}_{x,y}}(P)$.
Take $\varepsilon^{*} >0$ so that $(a^{*}-\varepsilon^{*}, a + \varepsilon^{*}) \subset (a-\varepsilon, a+ \varepsilon)$. Let $N= N_{\varepsilon^{*} \slash 2}$ and $\tau = \rho_{x,P}^{H}(y^{N})$.

Let $U=\{ Q \in \InvLO_{x}^{H}(G) \cap \Cof_{x}^{H}(G) \: | \: \rho_{x,Q}^{H}(y^{N}) = \tau\}$.
$U$ is written as
\[ U = ( \mU_{x} \cap \mU_{x^{-\tau} y^{N}} \cap \mU_{y^{-N}x^{\tau}}) \cup  ( \mU_{x^{-1}} \cap \mU_{x^{\tau} y^{N}} \cap \mU_{y^{-N}x^{-\tau+1}}) \]
by using the open basis $\{\mU_{g}\}_{g \in G}$, thus $U$ is open.
Now for $Q \in U$, 
\[ \left| \, \widetilde{\psi^{H}_{x,y}}(Q) - \widetilde{\psi^{H}_{x,y}}(P) \,\right|
\leq \left| \, \widetilde{\psi^{H}_{x,y}}(Q) - \frac{\rho_{x,Q}^{H}(y^{N})}{N} \, \right| + \left| \,\frac{\rho_{x,P}^{H}(y^{N})}{N}- \widetilde{\psi^{H}_{x,y}}(P) \,\right| < \varepsilon^{*}\]
so $ \widetilde{\psi^{H}_{x,y}}(U) \subset (a^{*} - \varepsilon^{*}, a^{*} + \varepsilon^{*})$.
Thus, $U$ is an open neighborhood of $P$ in $\widetilde{\psi^{H}_{x,y}}{}^{-1}(a - \varepsilon, a + \varepsilon)$, so $\widetilde{\psi^{H}_{x,y}}{}^{-1}(a - \varepsilon, a + \varepsilon)$ is open. Thus, we conclude that $\widetilde{\psi_{x,y}^{H}}|_{A}$ is continuous.

Finally, we show that $\widetilde{\psi_{x,y}^{H}}$ is continuous at a point in $ \InvLO_{x}^{H}(G) -(\InvLO_{x}^{H}(G) \cap \Cof_{x}^{H}(G))$. Let $U_{N} = \{ a \in \R \cup \{\infty \} \: | \: |a| > N  \}$.
Then,
\[ \widetilde{\psi_{x,y}^{H}} {}^{-1} (U_{N}) = \left \{ \mU'_{x} \cap ( \mU'_{x^{-N}y} \cup \mU'_{y^{-1}x^{-N}}) \right\} \cup  \left \{ \mU'_{x^{-1}} \cap ( \mU'_{x^{N}y} \cup \mU'_{y^{-1}x^{N}}) \right\}  \]
 where $\mU'_{g} = \mU_{g} \cap \InvLO_{x}^{H}(G)$.
 Thus, $\widetilde{\psi_{x,y}^{H}} {}^{-1}(U_{N})$ is open, so we conclude that  $\widetilde{\psi_{x,y}^{H}}$ is continuous.

\end{proof}

\begin{exam}
\label{exam:Sikora1}
The above abstract constructions of the maps from $\LO(G)$ to $S^{1}$ are motivated to understand the topology of $\LO(G)$, and it is a generalization of the map $f:\LO(\Z^{2}) \rightarrow S^{1}$ constructed in \cite{s}, which we describe here.

First we review a description of $\LO(\Z^{2})$ given by Sikora \cite{s}.
Let $a,b$ be the free generator of $\Z^{2}$.
Let $\R_{[)}$ be the real line with the topology having a basis of the form $[a,b)$, and let $\R_{(]}$ be the real line with the topology having a basis of the form $(a,b]$. Let $S^{1}_{[)} = \R_{[)} \slash \Z$ and $S^{1}_{(]} = \R_{(]} \slash \Z$, and regard them as a subset of $\R^{2}$. A point $(p,q) \in S^{1}$ is called {\it rational} if $p\slash q \in \Q$ or $q=0$.  
Let $X$ be the topological space obtained by identifying the irrational points of $S^{1}_{(]}$ and $S^{1}_{[)}$. Sikora showed that $X$ is homeomorphic to $\LO(\Z^{2})$. This description is useful, because such coordinates reflects the properties of the ordering, as we will see Example \ref{exam:Sikora2}.

Now let us define the continuous map $f: X=\LO(\Z^{2}) \rightarrow S^{1}$ by sending a point of $x \in S^{1}_{[)}$ or $S^{1}_{(]}$ to the corresponding point of $S^{1}$. 
The relationship between $\widetilde{\psi_{a,b}}$ and $f$ is as follows.
Let us take a double cover $\pi_{2}: S^{1} \rightarrow S^{1} = \R \cup {\infty}$, which sends $(p,q) \mapsto q\slash p$ (Here we regard $\pm 1 \slash 0$ as $\infty$).
Then $\pi_{2} \circ f = \widetilde{\psi_{a,b}}$ holds.
\end{exam}

Using almost the same arguments, we can prove the following generalized version of the proposition. 

\begin{thm}
\label{thm:continuous}
If $H$ is a finitely generated subgroup of $G$, whose $\R$-coefficient 2nd bounded cohomology vanishes, then 
$\psi_{x}^{H}: \InvLO_{x}^{H}(G) \cap \Cof_{x}^{H}(G) \rightarrow H_{b}^{2}(H;\Z) = (S^{1})^{b}$ and $\widetilde{\psi_{x}^{H}}: \InvLO_{x}^{H}(G) \rightarrow S^{b}$ are continuous.
\end{thm}
\begin{proof}
As in the proof of Theorem \ref{thm:surj}, it is sufficient to show the case $H$ is a free abelian group, but this case easily follows from the arguments of the proof of Proposition \ref{prop:continuous}.
\end{proof}

\section{Actions on real lines and circles}

 As is mentioned in many articles (See \cite{g}, \cite{n}), left orderings of groups are closely related to the group action on the real line and circles.  In this section we provide relations between our quasi-morphisms and dynamics of $G$ which are related to the left ordering. In this section, we always assume that $G$ is a countable group.
 
First we review the definition of the bounded Euler class, which contains much information about the group actions on circle.  
Let us consider the central extension
\[ 1 \rightarrow \Z \rightarrow \widetilde{\textrm{Homeo}_{+}}(S^{1}) \stackrel{p}{\rightarrow} \textrm{Homeo}_{+}(S^{1}) \rightarrow 1\]
 where 
$\widetilde{\textrm{Homeo}_{+}}(S^{1}) =\{\widetilde{f} \in \textrm{Homeo}_{+}(\R) \: | \: \widetilde{f}(x+1) = \widetilde{f}(x) + 1 \;\;(\forall x \in \R) \}$. 
Take a set-theoretical section $\sigma$ to $p$ by 
$\sigma(f) = \widetilde{f}$, where $\widetilde{f}$ is an unique element of $p^{-1}(f)$ such that $\widetilde{f}(0)\in [0,1)$, and let $c$ be the cochain of $\textrm{Homeo}_{+}(S^{1})$, defined by $c(f,g) = \sigma(fg)^{-1}\sigma(f)\sigma(g)$. Since $c(f,g) = \sigma(fg)^{-1}\sigma(f)\sigma(g)$ lies in the kernel of $p$ = $\Z$, hence we may regard $c$ as a $\Z$-valued bounded $2$-cocycle of $\textrm{Homeo}_{+}(S^{1})$.
The {\it bounded Euler class} is a $2$nd bounded cohomology class $[eu]$ of $\textrm{Homeo}_{+}(S^{1})$ defined by $c$.
For a group action on circle $\phi: G \rightarrow \textrm{Homeo}_{+}(S^{1})$, we define the {\it bounded Euler class} of the action $\phi$ by $eu(\phi) = \phi^{*}([eu]) \in H_{b}^{2}(G;\Z)$.

Next we review the relationships between orderings and dynamics.
Let $G \rightarrow \textrm{Homeo}_{+}(\R)$ be a faithful action of a countable group $G$ on the real line.
By choosing a countable dense sequence $\{r_{n}\}_{n >0}$ of $\R$, we define the left-ordering $<$ of $G$ by defining $g < g'$ if and only if the sequence of reals $\{g(r_{n})\}$ is bigger than the sequence $\{g'(r_{n})\}$ with respect to the lexicographical ordering of $\R^{\mathbb{N}}$.
 
Conversely, for an left ordering $P \in \LO(G)$, we can construct a faithful action of $G$ on the real line $A_{P}: G \rightarrow \textrm{Homeo}_{+}(\R)$, which we call a {\it dynamical realization} of the ordering $P$ as follows.

Take a numbering $\{g_{i}\}_{i\geq 0}$ of $G$, and define the real number $t(g_{i})$ in the following inductive way. First we define $t(g_{0}) = 0$. For $i \geq 1$, we define $t(g_{i})$ by
\[ t(g_{i}) = \left\{ \begin{array}{ll}
 \max \;\{ t(g_{0}), \ldots, t(g_{i-1})\} + 1 \;\;\;\;  & g_{i} >_{P} \max\{ g_{0},\ldots,g_{i-1}\}\\
 \min  \;\{ t(g_{0}), \ldots, t(g_{i-1})\} - 1 \;\;\;\;   & g_{i} <_{P} \min\{ g_{0},\ldots,g_{i-1}\} \\
 ( t(g_{m})+t(g_{M}))\slash 2 \;\;\;\; & g_{m} <_{P} g_{i} <_{P} g_{M}  \textrm{ and } \\
 & \;\;\;(g_{m},g_{M}) \cap \{g_{1},\ldots, g_{i-1}\} = \phi 
\end{array}
\right.
\]
Here $(g_{m},g_{M})$ is a subset $\{ g \in G \: | \: g_{m} <_{P} g <_{P} g_{M} \}$.

Now we define the action of $G$ on the subset $\{ t(g_{i})\}$ of $\R$ by $g \cdot t(g_{i}) = t(gg_{i})$. By extending this action to whole of $\R$, we obtain the desired action $A_{P}$.

Although this construction depends on a choice of the numbering and extensions, but its topological conjugacy class of the dynamical realization is independent of these choices.
From now on, we always choose a numbering so that $g_{0} = id$. Then by construction, $f <_{P} g$ if and only if $[A_{P}(f)] (0) < [A_{P}(g)](0)$.

Let $H$ be a subgroup of the left orderable group $G$ and take $x \in Z_{G}(H)$, the centralizer of $H$ in $G$. 
For $P \in \Cof_{x}^{H}(G)$, let us consider its dynamical realization $A_{P}$. 
We may assume that $A_{P}(x)$ acts $\R$ as translation $r \rightarrow r+1$ by taking an appropriate conjugation.   We denote this normalized dynamical realization by $\widetilde{A_{P,x}}$.
Since $x \in Z_{G}(H)$, $[\widetilde{A_{P,x}}(h)] (r + 1) =[\widetilde{A_{P,x}}(h)](r) + 1 $ holds for all $h\in H$ and $r \in \R$. Therefore the restriction of $\widetilde{A_{P,x}}$ to $H$ induces an action of $H$ on $\R \slash \langle \widetilde{A_{P,x}}(x) \rangle = S^{1}$. We denote this action by $A_{P,x}: H \rightarrow \textrm{Homeo}_{+} (S^{1})$ and call it the {\it associated circle action}.

Now we provide a cohomological interpretation of our map $\psi_{x}^{H}$ via the associated circle action. This theorem implies, $\psi_{x}^{H}$ can be seen as a characteristic class of left orderings, which extends the bounded Euler class of circle actions.
 
\begin{thm}
\label{thm:dynamics}
$ \psi_{x}^{H}(P) = - eu(A_{P,x})$. 
\end{thm}
\begin{proof}
For $f,g \in H$, let $F = A_{P,x}(f)$ and $G = A_{P,x}(g)$. By definition of the section $\sigma$, $\sigma(F)(0) = [\widetilde{A_{x,P}}(f)](0) - \rho_{x,P}(f)$ holds. 
 Since $\sigma(FG)^{-1}\sigma(F)\sigma(G) = [\sigma(FG)^{-1}](0) + [\sigma(F)](0) + [\sigma(G)](0)$, we conclude that 
 \[ \sigma(FG)^{-1}\sigma(F)\sigma(G) = -\rho_{x,P}(fg) + \rho_{x,P}(f) + \rho_{x,P}(g) .\]
Thus, by definition of the bounded Euler class, we have $ \psi_{x}^{H}(P) = - eu(A_{P})$.
\end{proof}

Recall that a left ordering $P \in \LO(G)$ is {\it dense} if $P$ admits no minimal positive element. For a dense ordering $P$, from the construction of the dynamical realization, all orbits of the dynamical realization are dense.  
We say two dense left orderings are {\it dynamically equivalent} if their dynamical realizations are conjugate. Recall that $\phi,\psi \in  \textrm{Homeo}_{+} (\R)$ are conjugate if there exists $f \in  \textrm{Homeo}_{+} (\R)$ such that $\phi \circ f = f \circ \psi $ holds.

\begin{exam}
The group $G$ itself naturally acts on $\LO(G)$ by $g: P \rightarrow P \cdot g$. 
Two left orderings are called {\it conjugate} if they belongs to the same $G$-orbit.
Let $P, Q$ be two left orderings which are conjugate by $h\in G$, and take a dynamical realization $A_{P}$ of $P$.
Then we can choose the dynamical realizations $A_{Q}$ as $A_{Q}(g) = A_{P}(h^{-1}gh)$, so they are dynamically equivalent.
\end{exam}

Now we show that $\psi_{x}$ classify the dynamically equivalence.

\begin{thm}
\label{thm:clasify}
Let $x \in Z(G)$ and $P,Q \in \Cof_{x}(G)$ be dense orderings.
Then $P$ and $Q$ are dynamically equivalent if and only if $\psi_{x}(P) = \psi_{x}(Q)$.
\end{thm}
\begin{proof}
First observe that if $P$ and $Q$ are dynamically equivalent then $A_{P,x}$ and $A_{Q,x}$ are conjugate.
Conversely, assume that $A_{P,x}$ and $A_{Q,x}$ are conjugate by $\phi$. Then the lift  $\sigma(\phi)$ provides a conjugation between $\widetilde{A_{P,x}}$ and $\widetilde{A_{Q,x}}$.
Since $P$ and $Q$ are dense, the all orbits of their associated circle actions $A_{P,x}$ and $A_{Q,x}$ are dense on $S^{1}$. By Ghys' theorem \cite[Theorem 6.5]{g}, $A_{P,x}$ and $A_{Q,x}$ are conjugate if and only if $eu(A_{P,x}) = eu(A_{Q,x})$. Thus by Theorem \ref{thm:dynamics} we obtain the desired result.
\end{proof}

Using this, we give a cohomological characterization of infinite Thurston type orderings of the braid groups. Thurston-type orderings are orderings of the braid group $B_{n}$ which are defined by the Nielsen-Thurston action, which is the action of $B_{n}$ on the real line $\R$ constructed by using Hyperbolic geometry. These orderings are generalization of the standard Dehornoy ordering $<_{D}$.  See \cite{ddrw},\cite{sw} for details.

\begin{cor}
A dense left ordering $P$ of the $n$-braid group $B_{n}$ is a Thurston type ordering if and only if $\psi_{\Delta^{2}}(P) = \psi_{\Delta^{2}}(D)$, where $D$ represents the Dehornoy ordering.
\end{cor}

\begin{rem}
We may define the notion of {\it semi-dynamical equivalence} of general left orderings by using the notion of semi-conjugation instead of conjugation. This relation is the same as the dynamical equivalence if we consider dense orderings. Since the bounded Euler class is a complete invariant of semi-conjugacy \cite{g}, we can rephrase Theorem \ref{thm:clasify} as 
{\it Two left orderings $P, Q \in \Cof_{x}(G)$ are semi-dynamically equivalent if and only if $\psi_{x}(P) = \psi_{x}(Q)$}.
\end{rem}

 \section{Convexity criterion}
 
In this section, we utilize the maps $\psi_{x}^{H}$ and $\rho_{x,P}^{H}$ to study the structure of the orderings.
Let $A \subset B$ be subsets of the left-orderable group $G$.
We say $A$ is {\it convex} in $B$ with respect to the ordering $P \in \LO(G)$ if some elements $b \in B$ satisfies the inequality $a <_{P} b <_{P} a'$ for some $a,a' \in A$, then $b \in A$.
Convex subgroups have the following properties.

\begin{lem}
\label{lem:convexproperty}
Let $G$ be the left-orderable group and $B$ be a subgroup of $G$.
Let $A,A'$ be subgroup of $B$, which are convex in $B$ with respect to an ordering $P \in \LO(G)$. Then either $A \subset A'$ or $A' \subset A$ holds. 
\end{lem}
\begin{proof}
Assume that we have neither  $A \subset A'$ nor $A' \subset A$.
Then there exist elements $a \in A- A'$ and $a' \in A' -A$. With no loss of generalities, we may assume $1<_{P} a$ and $1<_{P} a'$. 
If $a' <_{P} a$, then by convexity of $A$, $a' \in A$. Thus $a <_{P} a'$. On he other hand, by the same reason, we have $a' <_{P} a$. This is a contradiction.
\end{proof}

The following lemma is simple but provides a useful criterion for convexity of some subgroups. 

\begin{lem}
\label{lem:psivalue}
Let $A$ be a subgroup of $G$ generated by $x_{1},\ldots,x_{m}$.
Let $w$ be an element of $A$, and fix a word expression $w=x_{i_{1}}^{j_{1}}x_{i_{2}}^{j_{2}}\cdots x_{i_{n}}^{j_{n}}$ and define $e_{p} = \sum_{q\:;\: i_{q}=p} j_{q}$.
For an ordering $P \in \InvLO_{x}^{A}(G)$, if $\rho_{x,P}^{A}(w^{N}) = 0$ holds for all $N \in \Z$ (This assumption is satisfied, for example, $\langle w \rangle$ is convex in $A$), then
\[ \left| \sum_{q=1}^{n} e_{q} \overline{\rho_{x,P}^{A}}(x_{q}) \right| \leq n\]
 holds. Moreover, if $A$ is an abelian group, then
\[ \sum_{q=1}^{n} e_{q} \overline{\rho_{x,P}^{A}}(x_{q}) = 0 \]
 holds.
\end{lem}

\begin{proof}
By Lemma \ref{lem:stablemap}, we have 
\[ \left| \sum_{q=1}^{n} N e_{q} \overline{\rho_{x,P}^{A}}(x_{q}) \right|  \leq nN -1 \]
holds for all $N>0$, thus we conclude that 
\[ \left| \sum_{q=1}^{n}e_{q} \overline{\rho_{x,P}^{A}}(x_{q}) \right|  \leq n - \frac{1}{N}  \]
holds for all $N>0$. 
Similarly, if $A$ is abelian, then $w^{N} = x_{1}^{Ne_{1}}x_{2}^{Ne_{2}}\cdots x_{m}^{Ne_{m}}$, so we have the inequality 
\[ \left| \sum_{q=1}^{n}e_{q} \overline{\rho_{x,P}^{A}}(x_{q}) \right|  \leq \frac{m-1}{N}. \]
for all $N >0$.
\end{proof}

\begin{exam}
Let $B_{3}=\langle x,y \: | \: x^{2} = y^{3}\rangle$ be the braid group of $3$-strands.
Let us consider the cyclic subgroup $A= \langle w \rangle$, where $w=xy^{2}=x^{3}y^{-1}$.
Assume that there exists a left-ordering $P \in \InvLO_{x}^{A}(G)$ which makes $A$ convex.
Then, by Lemma \ref{lem:psivalue} we have two inequalities
\[ |1+ 2\overline{\rho_{x,P}^{B_{3}}}(y)| \leq 2, \; \; |3 -\overline{\rho_{x,P}^{B_{3}}}(y)| \leq 2. \]
However, there is no real number which satisfies the above two inequalities, so we conclude that $\langle  w \rangle$ is not convex.
\end{exam} 

Now we provide a criterion for convexity for abelian subgroups.

\begin{thm}
\label{thm:convex}
Let $G$ be a left orderable group and $x \in G$.
Let $A$ be a rank $n$ free abelian subgroup of $G$ generated by $x_{1},\ldots,x_{n}$ and $B$ be a rank $k(<n)$ free abelian subgroup of $A$ generated by 
$\{ w_{i}= x_{1}^{e_{1}^{i}}\cdots x_{n}^{e_{n}^{i}} \}_{i= 1,\ldots,k}$. 
Then, the subgroup $B$ is convex in $A$ with respect to the ordering $P \in \InvLO_{x}^{A}(G) \cap \Cof_{x}^{A}(G)$ if and only if the following three conditions are satisfied.
\begin{enumerate}
\item Integers $e^{i}_{1},\ldots,e^{i}_{n}$ have no common divisor for all $i = 1,\ldots,k$.
\item $\sum_{j} e_{j}^{i}\overline{\rho_{x,P}^{A}}(x_{j}) = 0$ holds for all $i = 1,\ldots,k$.
\item $\dim_{\Q} \textrm{span}_{\Q}\{ \overline{\rho_{x,P}^{A}}(x_{1}),\ldots, \overline{\rho_{x,P}^{A}}(x_{n}) \} = n-k$.
\end{enumerate} 
\end{thm}

\begin{proof}
We prove the theorem by induction of $n$.
The case $n=1$ is obvious. Assume that we have already proved the theorem for $<n$.

Assume that $B$ is convex. Then, (1) is obvious, and the assertion (2) follows from Lemma \ref{lem:psivalue}. Assume that (3) does not hold.
Then, we can find other linearly independent linear relations of $\{ \overline{\rho_{x,P}^{A}}(x_{1}),\ldots, \overline{\rho_{x,P}^{A}}(x_{m}) \}$.
Thus, by choosing a re-numbering of $\{x_{i}\}$ if necessary, there exists an integer $m<n$ and integers $\{q^{i}_{j}\}_{j=1,\ldots,m}$ such that 
\[ \sum_{j=1}^{m} q^{i}_{j}\, \overline{\rho_{x,P}^{A}}(x_{j}) = 0, \textrm{ and  } 
 \dim_{\Q} \textrm{span}_{\Q} \{ \overline{\rho_{x,P}^{A}}(x_{1}),\ldots, \overline{\rho_{x,P}^{A}}(x_{m}) \} =m-k \]
hold. We can choose $\{q^{i}_{j}\}_{j=1,\ldots,m}$ so that they have no common divisors for all $i$. 
Let $A'$ be the subgroup of $A$ generated by $\{x_{1},\ldots,x_{m}\}$ and $B'$ be the subgroup of $A'$ generated by $\{ x_{1}^{q_{1}^{i}} \cdots x_{m}^{q_{m}^{i}}\}_{i=1,\cdots,k}$.
Then by induction, $B'$ is convex in $A'$. 
By definition of the subgroup $B'$, $B' \not \subset B\cap A'$ and $B \cap A' \not \subset B'$ holds. 
On the other hand, since $B$ is convex, $B\cap A'$ is a convex subgroup of $A'$.
This contradicts Lemma \ref{lem:convexproperty}.

Conversely, assume that all conditions {\it 1-3} hold but $B$ is not convex. Then, there exists an element $y = x_{1}^{q_{1}}\cdots x_{n}^{q_{n}}\not \in B$ such that $b <_{P} y <_{P} b'$ holds for some $b, b' \in B$.
With no loss of generality, we may assume that the integers $\{q_{i}\}$ have no common divisor.
Now, $\rho_{x,P}^{A}(y^{N}) = 0$ for all $N \in \Z$, so by Lemma \ref{lem:psivalue}, $\sum q_{i}\overline{\rho_{x,P}^{A}}(x_{i}) = 0 $.
Since $y \not \in B$, we have $\dim_{\Q} \textrm{span}_{\Q}\{ \overline{\rho_{x,P}^{A}}(x_{1}),\ldots, \overline{\rho_{x,P}^{A}}(x_{n}) \} < n-k$, which contradicts the condition (3).

\end{proof}

This theorem is useful to relate the topology of $\LO(G)$ and the structure of orderings. In particular, this implies that $\widetilde{\psi_{x}^{H}}$ serves as some kinds of good coordinates, which reflects the property of orderings.
The following observation is first done by Adam Clay, by using Sikora's description of $\LO(\Z^{2})$, which provides relationships between the topology of $\LO(\Z^{2})$ and convex subgroups. 

\begin{exam}
\label{exam:Sikora2}
Let us consider the rank two free abelian group $A= \langle x,y \rangle$. Then by Theorem \ref{thm:convex}, the subgroup $\langle  x^{p}y^{q}\rangle$ is convex with respect to the ordering $P \in \InvLO_{x}^{A}(A)$ if and only if 
$p,q$ has no common divisor and $\overline{\rho_{x,P}^{A}}(y) = -q \slash p$.
On the other hand, if $P \not \in \InvLO_{x}^{A}(A)$, then $\langle x\rangle$ is convex in $A$ with respect to the ordering $<_{P}$.
 
Thus, an ordering $P \in \LO(A)$ admits a convex non-trivial subgroup if and only if $\widetilde{\psi^{A}_{x,y}}(P) \in \Q \cup \{ \infty \}$. In Sikora's description of the space $\LO(G)$, this implies that $P \in \LO(A)$ admits a convex non-trivial subgroup if and only if $P$ is a rational point of $S^{1}_{[)}$ or $S^{1}_{(]}$.
\end{exam}

\vspace{1cm}
\begin{flushright}
\begin{minipage}{.5\textwidth}
\noindent
Tetsuya Ito \\
University of Tokyo\\
3-8-1 Komaba Meguro-Ku, Tokyo \\
Japan\\
e-mail: tetitoh@ms.u-tokyo.ac.jp
\end{minipage}
\end{flushright}

\end{document}